\documentclass{article}
\pdfoutput=1
\usepackage{float}
\usepackage[english]{babel}
\usepackage[margin=1in]{geometry}
\usepackage{background}
\usepackage{amssymb}
\usepackage{amsthm}
\usepackage{dsfont}
\usepackage{amsmath, amsfonts}
\usepackage{scrextend}
  \usepackage{paralist}
  \usepackage{graphics} 
  \usepackage{epsfig} 
\usepackage{graphicx}
\usepackage{epstopdf}
\usepackage[font={small,it}]{caption}
\usepackage{subcaption}
\graphicspath{{figures/}} 

 \usepackage[colorlinks=true]{hyperref}
\hypersetup{urlcolor=blue, citecolor=red}

\SetBgContents{}

\DeclareMathOperator{\id}{id}

\DeclareMathOperator{\N}{N}

\DeclareMathOperator{\supp}{supp}

\newcommand{\sign}{\text{sign}}

\newtheorem{theorem}{Theorem}[section]
\newtheorem*{theorem-non1}{Theorem 1.1}

\newtheorem{lemma}[theorem]{Lemma}
\newtheorem{proposition}[theorem]{Proposition}

\theoremstyle{definition}

\newtheorem{remark}[theorem]{Remark}

\newcommand{\rmd}{\mathrm{d}}
\renewcommand{\epsilon}{\varepsilon}

\newcommand{\fa}          {\quad \text{for all } \,}

\numberwithin{equation}{section}

\begin{document}

%
%
%
%
%

\title{\vspace*{-10mm}
Bifurcation analysis of a stochastically driven limit cycle}
\author{Maximilian Engel\thanks{Zentrum Mathematik der TU M\"{u}nchen,
Boltzmannstr. 3, D-85748 Garching bei M\"{u}nchen} \and Jeroen S.W.~Lamb\thanks{Department of Mathematics, Imperial College London, 180 Queen’s Gate, London SW7 2AZ, United Kingdom
} \and Martin~Rasmussen\footnotemark[2]}
\date{\today}
\maketitle

\bigskip

\begin{abstract}
We establish the existence of a bifurcation from an attractive random equilibrium to shear-induced chaos for a stochastically driven limit cycle, indicated by a change of sign of the first Lyapunov exponent. This addresses an open problem posed by Kevin Lin and Lai-Sang Young in \cite{ly08, y08}, extending results by Qiudong Wang and Lai-Sang Young \cite{wy03} on periodically kicked limit cycles to the stochastic context.

\bigskip

\noindent $\textit{Key words}$. Furstenberg--Khasminskii formula, Lyapunov exponent, Random dynamical system, shear-induced chaos, stochastic bifurcation

\noindent $\textit{Mathematics Subject Classification (2010)}$.
37D45, 37G35, 37H10, 37H15.

\end{abstract}

\section{Introduction} \label{intro}
We consider the following model of a stochastically driven limit cycle
\begin{align} \label{model1}
  \begin{array}{r@{\;\,=\;\,}l}
    \rmd y & - \alpha y \,\rmd t + \sigma \sum_{i=1}^m  f_i(\vartheta) \circ \rmd W_t^i\,,\\
    \rmd \vartheta & (1 + b y) \,\rmd t\ 
 \,,
  \end{array}
\end{align}
where $(y, \vartheta)\in\mathbb{R}\times \mathbb S^1$ are cylindrical amplitude-phase coordinates, $m \in \mathbb N$, and $W_t^i$ for $i\in\{ 1, \dots, m\}$,
denote independent
one-dimensional Brownian motions entering the equation as noise of Stratonovich type. In the absence of noise ($\sigma=0$), the ODE (\ref{model1}) has a globally attracting limit cycle at $y=0$ if $\alpha > 0$. In the presence of noise ($\sigma \neq 0$), the amplitude is driven by phase-dependent noise. The real parameter $b$ induces shear: if $b \neq 0$, the phase velocity $\frac{\rmd\vartheta}{\rmd t}$ depends on the amplitude $y$.

The stable limit cycle turns into a random attractor if $\sigma\not=0$. The main question we address in this paper concerns the nature of this random attractor. The crucial quantity is the sign of the first Lyapunov exponent $\lambda_1$ with respect to the invariant measure associated to the random attractor. In essence, $\lambda_1$ is the dominant infinitesimal asymptotic expansion rate of almost all trajectories.

To facilitate the analysis, we choose $f_i: \mathbb S^1\simeq [0,1) \to \mathbb{R}$ such that
\begin{equation} \label{sumcondition}
 \sum_{i=1}^m f_i'(\vartheta)^2 = 1 \ \fa \vartheta \in \mathbb S^1\,.
\end{equation}
If $m \geq 2$, then the functions are assumed to be smooth.
The simplest example is given by
\begin{equation}  \label{fdef2}
m=2, \quad f_1(\vartheta) = \cos(\vartheta), \quad f_2(\vartheta) = \sin(\vartheta)\,.
\end{equation}
For $m=1$, condition~\eqref{sumcondition} cannot be satisfied for all $\vartheta \in \mathbb S^1$. Hence, we choose $f := f_1 : \mathbb S^1\simeq [0,1) \to \mathbb{R}$ to be continuous and piecewise linear with constant absolute value of the derivative almost everywhere. The simplest example is given by
\begin{equation}  \label{fdef}
f(\vartheta) = \begin{cases}  \vartheta \ & \text{if} \ \vartheta \leq \frac{1}{2}\,, \\(1-\vartheta)\ &  \text{if} \ \vartheta \geq \frac{1}{2}\,. \end{cases}
\end{equation}
With such choices of the amplitude-phase coupling we obtain the following bifurcation result.
\begin{theorem} \label{maintheorem}
Consider the SDE \eqref{model1} with $f_i$, $i=1,\dots,m$, satisfying condition~\eqref{sumcondition}. Then there is $c_0 \approx 0.2823$ such that for all $\alpha > 0$ and $b \neq 0$ the number  $ \sigma_{0}(\alpha, b) = \frac{\alpha^{3/2}}{c_0^{1/2} \left|b\right|} > 0$ is the unique value of $\sigma$ where the top Lyapunov exponent $\lambda_1(\alpha, b,\sigma)$ of \eqref{model1} changes its sign. In more detail, we have
\begin{equation*}
\lambda_1(\alpha, b,\sigma) \begin{cases}
& < 0 \quad \text{if} \ 0 < \sigma < \sigma_0(\alpha, b)\,, \\
& = 0 \quad \text{if} \ \sigma = \sigma_0(\alpha, b)\,, \\
& > 0 \quad \text{if} \ \sigma > \sigma_0(\alpha, b)\,.
\end{cases}
\end{equation*}
\end{theorem}
As long as $b,\sigma\neq0$, the amplitude variable $y$ can be rescaled so that the shear parameter becomes equal to $1$ and the effective noise-amplitude becomes $\sigma b$.  Hence, the above result also holds with the roles of $\sigma$ and $b$ interchanged. The fact that $\sigma_{0}(\alpha, b)$ is an increasing function of $\alpha$ is illustrated in Figure~\ref{fig1}. Figure~\ref{fig3d} depicts $\lambda_1$ as a function of $\alpha$ and $\sigma$ for fixed $b=2$. Figure~\ref{fig2d} displays the corresponding areas of positive and negative top Lyapunov exponent in the $(\sigma,\alpha)$-parameter space where $\lambda_1 =0$ along the curve $\{(\sigma_0(\alpha,2), \alpha)\}$ separating the two areas. If $\sigma =0$, we clearly have $\lambda_1 = 0$ for all $\alpha > 0$. The case $\alpha = 0$ is obviously not of any interest in our model.
\begin{figure}[H]
\begin{subfigure}{.5\textwidth}
  \centering
  \includegraphics[width=1\linewidth]{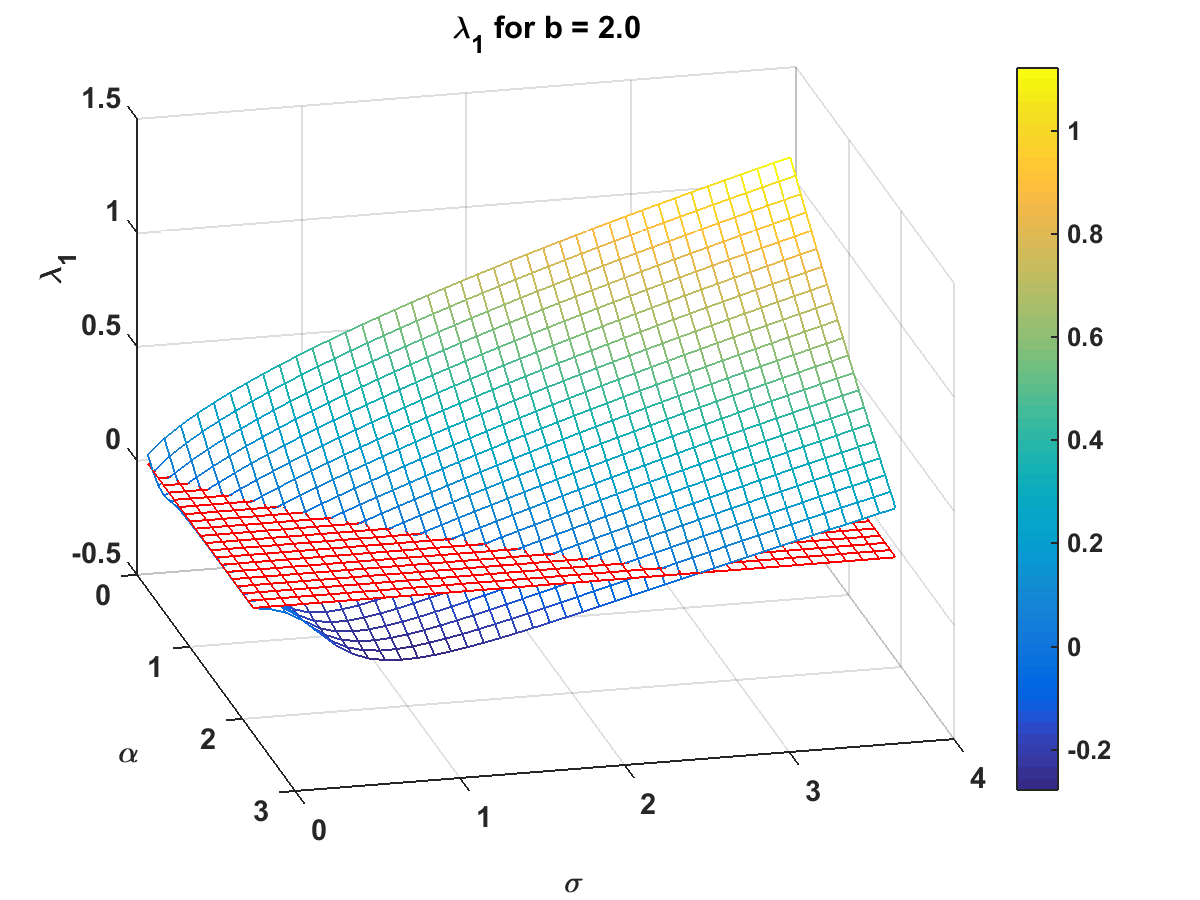}
\caption{}
 \label{fig3d}
\end{subfigure}
\begin{subfigure}{.5\textwidth}
  \centering
  \includegraphics[width=1\linewidth]{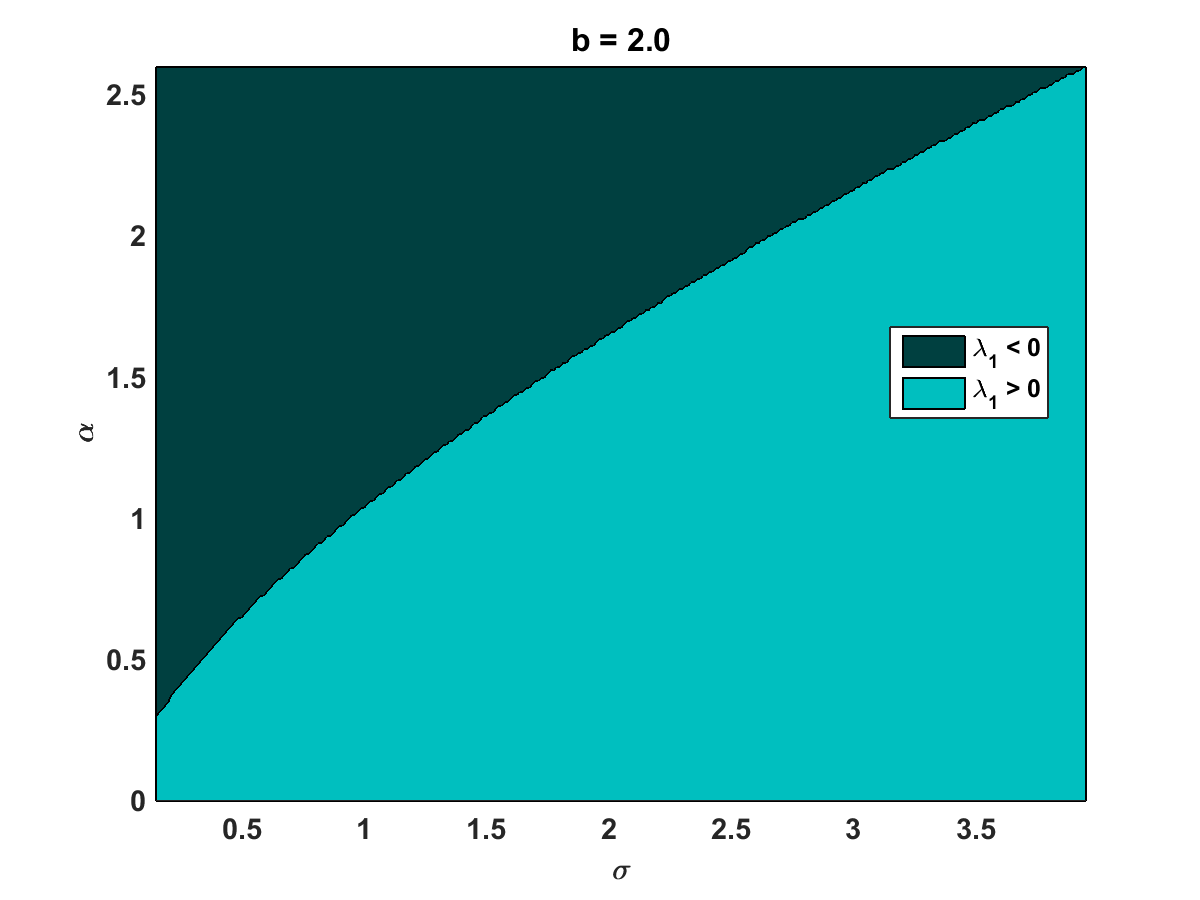}
   \caption{}
  \label{fig2d}
\end{subfigure}
\caption{In Figure~\ref{fig3d} we depict the top Lyapunov exponent $\lambda_1$, computed using (\ref{topLyap}), as a function of  $\alpha$ and $\sigma $ for fixed $b=2$. The red mesh demarcates the level $\lambda_1 =0$. Figure~\ref{fig2d} shows the corresponding areas of positive and negative $\lambda_1$ in the $(\sigma,\alpha)$-parameter space being separated by the curve $\{(\sigma_0(\alpha,2), \alpha)\}$. The picture doesn't display $\sigma =0$: in this case, $\lambda_1 = 0$ for all $\alpha > 0$.}
\label{fig1}
\end{figure}
If the top Lyapunov exponent is negative, it turns out that the (weak) random point attractor is an \emph{attracting random equilibrium}, i.e.~its fibers are singletons almost surely. Properties of random attractors with positive top Lyapunov exponents are not yet well understood, apart from the fact that such attractors are not random equilibria. They are sometimes referred to as \emph{random strange attractors} \cite{ly88, w09}.
Theorem~\ref{maintheorem} confirms numerical results by Lin \& Young \cite{ly08} for a very similar model. The mechanism, whereby a combination of shear and noise causes stretching and folding leading to a positive Lyapunov exponent, was coined with the term \emph{shear-induced chaos} by Lin \& Young. Wang \& Young \cite{wy02,wy03} and  Ott \& Stenlund \cite{os10} have demonstrated analytically the validity of this mechanism in the case of periodically kicked limit cycles, including probabilistic characterizations of the dynamics. An analytical proof of shear-induced chaos in the stochastic setting, as presented in this paper, had remained an open problem.

The results of this paper are part of a larger effort to develop a bifurcation theory of random dynamical systems. Earlier attempts to develop such a theory (notably by Ludwig Arnold, Peter Baxendale and coworkers \cite{a98, ass96, b94, sh96} in the 1990s) led to notions of so-called \emph{phenomenological} (or "P") bifurcations and \emph{dynamical} (or "D") bifurcations, but there is growing evidence that these paradigms do not suffice to capture the intricacies of bifurcation in random dynamical systems \cite{acd16,cdlr16,Lamb_15_1,Zmarrou_07_1}. In the absence of a consensus on useful characterisations of the dynamics of random systems and bifurcations in this context, much of the current research inevitably focusses on the detailed analysis of relatively elementary examples, to generate insights and guidance towards the further development of a more general theory.

In one-dimensional SDEs, negative Lyapunov exponents and attractive random equilibria prevail \cite{cf98}. Random strange attractors can only arise in dimension two and higher and up to now, little research has been devoted to such attractors. In contrast, the existence of attractive random equilibria (also referred to as \emph{synchronization}, with reference to the corresponding dynamics of sets of initial conditions) has been studied well, also in higher dimensions \cite{b91, fgs16, lj87, nj15, nj16}.

The main technical challenge addressed in this paper is to establish the existence of positive top Lyapunov exponents. Most rigorous results on Lyapunov exponents (and random dynamical systems) are obtained for one-dimensional SDEs, in which case the analysis of Lyapunov exponents significantly simplifies due to the fact that all derivatives commute. It is difficult in general to obtain lower bounds for the top Lyapunov exponent in higher dimensions due to the subadditivity property of matrices, cf. \cite{y08}. Thus, the analytical demonstration of  positive Lyapunov exponents for noisy systems has been achieved only in certain special cases, like for equilibria \cite{IL99}, simple time-discrete models as in \cite{ls12} or under special circumstances that allow for the use of stochastic averaging \cite{b04, bg02}. In our setting, condition~\eqref{sumcondition} is crucial to establish rigorous lower bounds on the top Lyapunov exponent $\lambda_1$.

Another prototypical open problem in dimension two is the \emph{stochastic Hopf bifurcation}, concerning the characterisation of dynamics and bifurcations in parametrized families of SDEs that in the deterministic (noise-free) limit display a \emph{Hopf bifurcation}. A (deterministic) Hopf bifurcation occurs if, by the variation of a model parameter, an asymptotically stable equilibrium loses stability under the emission of a small attracting limit cycle. Numerical studies \cite{w09} suggest that the mechanism of shear-induced chaos is at play also in stochastic Hopf bifurcations, but while analytical proofs of parameter regimes with negative top Lyapunov exponents are within reach \cite{dsr11,delr17}, until now, there are no rigorous results concerning the existence of parameter regimes with positive top Lyapunov exponents in this context. The results of this paper may well be relevant to shed more light on this problem.

The remainder of the paper is organized as follows. Section \ref{modelsec} provides the analysis of Lyapunov exponents for our model: Subsection~\ref{GenLyap} introduces the model on the cylinder within the framework of random dynamical systems and establishes the necessary theoretical concepts. Subsection~\ref{FK} introduces the Furstenberg--Khasminskii forumula for the top Lyapunov exponent and in Subsection~\ref{Top Lyapunov}, we derive a formula for the top Lyapunov exponent $\lambda_1$. The main result concerning the change of sign of $\lambda_1$ is proven in Section \ref{bifurresult} and its consequences are discussed. We give illustrations of $\lambda_1$ in dependence on the parameters and confirm a scaling conjecture by Lin and Young.
We conclude with a short summary in Section \ref{summary}.

\section{Analysis of the top Lyapunov exponent} \label{modelsec}
\subsection{Lyapunov exponents for random dynamical systems} \label{GenLyap}
Consider the stochastic differential equation of Stratonovich type
\eqref{model1}. We assume that $f_i:[0,1] \to \mathbb{R}$, $i = 1, \dots, m$, are Lipschitz continuous functions with $f(0)=f(1)$ (smooth if $m \geq 2$ and piecewise linear if $m=1$), and the three parameters fulfill $\alpha >0$, $\sigma > 0$ and $b \in \mathbb R$. Note that the equation reads the same in It\^{o} form according to the It\^{o}--Stratonovich conversion formula.

Since the drift and diffusion coefficients are Lipschitz continuous and satisfy linear growth conditions, the SDE~\eqref{model1} generates a continuous random dynamical system $(\theta, \varphi)$ \cite[Definition 1.1.2]{a98} consisting of the following:
\begin{enumerate}
\item[(i)] A model of the noise on the probability space $\Omega := C_0(\mathbb{R}, \mathbb{R})= \{ \omega \in C(\mathbb{R}, \mathbb{R}): \omega(0) =0 \} $ with Borel $\sigma$-algebra $\mathcal{F}$ and two-sided Wiener measure $\mathbb{P}$, formalized as the family $(\theta_t)_{t \in \mathbb{R}}$ of $\mathbb P$-preserving shift maps given by
$(\theta_t \omega)(s) = \omega(s+t) - \omega(t)$.
\item[(ii)] A model of the system perturbed by noise formalized as a \textit{cocycle} $\varphi$ over $\theta$ of mappings of $\mathbb R \times \mathbb{S}^1$, i.e.~$\varphi$ is a $ \mathcal{B}(\mathbb{R}_0^+)\otimes \mathcal{F} \otimes \mathcal{B}(\mathbb R \times \mathbb{S}^1)$-measurable mapping
\begin{equation*}
\varphi: \mathbb{R}_0^+ \times \Omega \times (\mathbb R \times \mathbb{S}^1) \to \mathbb R \times \mathbb{S}^1, \quad (t, \omega, x) \mapsto \varphi(t, \omega)x,
\end{equation*}
such that $(t, x) \mapsto \varphi(t, \omega)x$ is continuous for every $\omega \in \Omega$ and which satisfies
\begin{displaymath}
  \varphi(0, \omega) = \id \quad \text{and} \quad \varphi(t+s, \omega) = \varphi(t, \theta_s \omega) \circ \varphi(s, \omega) \quad \text{ for all } \omega \in \Omega \text{ and } t, s \in \mathbb{R}_0^+\,.
\end{displaymath}
\end{enumerate}
The random dynamical system $(\theta,\varphi)$ induced by \eqref{model1} is also a skew product flow $\Theta = (\theta, \varphi)$, which is a measurable dynamical system on the extended phase space $\Omega\times X$.
The skew product flow $\Theta$ possesses an ergodic invariant Markov measure $\mu$ which is associated to the unique invariant measure (also called stationary measure) for the corresponding Markov semigroup.
Their existence follows from the same considerations as in \cite{ly08}.

Fundamental for stochastic bifurcation theory is Oseledets' Multiplicative Ergodic Theorem, which implies the existence of Lyapunov exponents describing stability properties of a differentiable random dynamical system. The random dynamical system $(\theta, \varphi)$ is called $C^k$ if $\varphi(t, \omega) \in C^k$ for all $t\in\mathbb R_0^+$ and $\omega\in\Omega$.  In the situation of the Stratonovich SDE
\begin{equation*}
\rmd X_t = f_0(X_t) \rmd t + \sum_{i=1}^m f_j(X_t) \circ \rmd W_t^j
\end{equation*}
on a smooth manifold $X$, the Jacobian $D \varphi(t,\omega,x)$ with respect to the third variable of the cocycle $\varphi(t,\omega)x$ is a linear cocycle over the skew product flow $\Theta=(\theta,\varphi)$. The Jacobian $D \varphi(t,\omega,x)$ applied to an initial condition $v_0 \in T_x X$ solves uniquely the variational equation on $T_x X \cong \mathbb{R}^d$, given by
\begin{equation}\label{varproblem}
  \rmd v = Df_0(\varphi(t,\omega)x) v \,\rmd t + \sum_{j=1}^m Df_j(\varphi(t,\omega)x) v \circ \rmd W_t^j \,, \quad \text{where } v \in T_x X\,.
\end{equation}
Suppose the one-sided $C^1$-random dynamical system $(\varphi, \theta)$ has an ergodic invariant measure $\nu$ and satisfies the integrability condition
\begin{equation*}
\sup_{0 \leq t \leq 1} \log^+ \| D \varphi(t, \omega, x) \| \in L^1(\nu).
\end{equation*}
Then the Multiplicative Ergodic Theorem for differentiable random dynamical systems  \cite[Theorem 3.4.1, Theorem 4.2.6]{a98} guarantees the existence of a $\Theta$-forward invariant set $\Delta \subset \Omega \times X$ with $\nu (\Delta) = 1$ and  the Lyapunov exponents $\lambda_1 > \dots > \lambda_p$ with respect to $\nu$.
The tangent space $T_x X \cong \mathbb{R}^d$ admits a filtration
\begin{displaymath}
  \mathbb R^d = V_1(\omega, x) \supsetneq V_2(\omega,x)\supsetneq \dots \supsetneq V_p(\omega,x) \supsetneq V_{p+1}(\omega,x)= \{0\}\,,
\end{displaymath}
such that for all $0 \neq v \in T_x X \cong \mathbb{R}^d$, the Lyapunov exponent $\lambda(\omega, x, v)$ defined by
\begin{equation*}
\lambda(\omega, x, v) = \lim_{t \to \infty} \frac{1}{t} \log \| D \varphi(t, \omega, x)v \|
\end{equation*}
exists and
\begin{equation*}
\lambda (\omega, x, v) = \lambda_{i}\quad \Longleftrightarrow \quad v \in V_i(\omega,x) \setminus V_{i+1}(\omega,x) \quad \text{ for all } i\in\{1, \dots, p\}\,.
\end{equation*}
%
\subsection{The Furstenberg--Khasminskii formula} \label{FK}

In the following, we calculate the top Lyapunov exponent $\lambda_1$ for the random dynamical system induced by \eqref{model1}. We consider the corresponding variational equation describing the flow on the tangent space $T_x(\mathbb{R}\times \mathbb S^1)\cong \mathbb{R}^2$ along trajectories of \eqref{model1}. The variational equation reads as
\begin{equation} \label{varEqu}
\rmd v = \underbrace{\begin{pmatrix}
- \alpha & 0 \\ b & 0
\end{pmatrix}}_{=:A} v \,\rmd t + \sigma \sum_{i=1}^m \underbrace{\begin{pmatrix}
0 &  f_i'(\vartheta) \\ 0 & 0
\end{pmatrix}}_{=:B_i} v \circ \rmd W_t^i\,.
\end{equation}
Note that we omit the $(t,\omega)$-dependence of $\vartheta$ and $B$. Because of the linearity of \eqref{varEqu}, we introduce the change of variables $r = \|v\|$ and $s = v/r$, so that $s$ lies on the unit circle. Its dynamics are given by
\begin{align*}
\rmd s &= ( As - \langle s, As \rangle s)\, \rmd t + \sum_{i=1}^m (B_is - \langle s, B_is \rangle s)\, \circ \rmd W_t^i \\
&= \begin{pmatrix}
- \alpha s_1 - s_1(-\alpha s_1^2 + b s_1 s_2) \\ b s_1 - s_2(-\alpha s_1^2 + b s_1 s_2)
\end{pmatrix} \rmd t + \sigma \sum_{i=1}^m \begin{pmatrix}
 f_i'(\vartheta) s_2 - s_1 f_i'(\vartheta) s_1 s_2 \\
-s_2  f_i'(\vartheta) s_1 s_2
\end{pmatrix} \circ \rmd W_t^i\,.
\end{align*}
The Furstenberg--Khasminskii formula for the top Lyapunov exponent \cite{IL99} is given by
\begin{equation}\label{fkformula}
\lambda_1 = \int_{\mathbb{R}}\int_{[0,1]}\int_{\mathbb S^1} ( h_A(s) + \sum_{i=1}^m k_{B_i}(s) )\, \rho(\rmd s, \rmd\vartheta, \rmd y),
\end{equation}
where $\rho$ is the joint invariant measure for the diffusion $s$ on the unit circle and the processes $\vartheta$ and $y$ induced by \eqref{model1}; the functions $h_A$ and $k_{B_i}$, $i = 1, \dots, m$, are given by
\begin{align*}
h_A(s) &= \langle s, As \rangle = - \alpha s_1^2 + b s_1 s_2\,,\\
k_{B_i}(s) &= \frac{1}{2} \langle \left( B_i + B_i^* \right)s, B_is \rangle - \langle s, B_is \rangle^2 = \frac{1}{2} \sigma^2 f_i'(\vartheta)^2 s_2^2 - \sigma^2 f_i'(\vartheta)^2 s_1^2 s_2^2\,.
\end{align*}
Similarly to the calculations in \cite{IL99}, we change variables to $s = (\cos \phi, \sin \phi)$. Note that the functions $h_A$ and $k_{B_i}$ are $\pi$-periodic, which implies that the formula \eqref{fkformula} for the top Lyapunov exponent reads as
\begin{equation} \label{GeneralLyap}
  \lambda_1 = \int_{\mathbb{R}\times[0,1]\times [0,\pi]} \left(  - \alpha \cos^2 \phi + b \cos \phi \sin \phi + \sum_{i=1}^m  f_i'(\vartheta)^2 \left[\frac{1}{2}\sigma^2 (1 - 2 \cos^2 \phi)\sin^2 \phi \right] \right) \tilde{\rho}(\rmd \phi, \rmd \vartheta, \rmd y),
\end{equation}
where $\tilde{\rho}$ denotes the corresponding image measure of $\rho$.
The SDE determining the dynamics of $ \phi \in [0, \pi)$ reads as
\begin{equation} \label{GeneralPhi}
  \rmd \phi = - \frac{1}{\sin \phi} \rmd s_1 = (\alpha \cos \phi \sin \phi + b \cos^2 \phi) \rmd t - \sum_{i=1}^m \sigma f_i'(\vartheta) \sin^2 \phi  \circ \rmd W_t^i\,,
\end{equation}
where we denote
\begin{equation}\label{cd}
  c_i(\phi, \vartheta) = \sigma f_i'(\vartheta) \sin^2 \phi \quad \text{and}\quad d(\phi) = \alpha \cos \phi \sin \phi + b \cos^2 \phi  \,.
\end{equation}
In the Fokker--Planck equation for $\phi$, the dependence on $\vartheta$ is restricted to $ \sum_{i=1}^m f_i'(\vartheta)^2$,
and in addition to that, the integrand of \eqref{GeneralLyap} only depends on $\phi$ and not on $\vartheta$ and $y$ if $ \sum_{i=1}^m f_i'(\vartheta)^2$ is constant. This means that the calculation of $\lambda_1$ becomes much simpler if $\sum_{i=1}^m f_i'(\vartheta)^2$ is constant, an observation that we exploit in the following.

\subsection{Explicit formula for the top Lyapunov exponent} \label{Top Lyapunov}

Firstly, we have to justify the analysis of the top Lyapunov exponent for the case $m=1$ where $f:=f_1: [0,1] \to \mathbb{R}$ is given by \eqref{fdef}. Importantly, $f'(\vartheta)^2$ is constant in this special case and our results hold in fact for every continuous and piecewise linear $f$ with constant absolute value of the derivative almost everywhere.

The map is not differentiable at $\frac{1}{2}$ and $0$, and we verify that does not cause any problems. We need the following results to justify the variational equation defining $D \varphi$:
\begin{lemma} \label{zerointegral}
  Let $W : \mathbb{R}^+_0 \times \Omega \to \mathbb{R}$ denote the canonical real-valued Wiener process, and let $X: \mathbb{R}^+_0 \times \Omega \to [0,1]$ be a stochastic process adapted to the natural filtration of the Wiener process. Furthermore, suppose there exists a measurable set $A \subset [0,1]$ such that
  \begin{equation}  \label{hittingtime}
    \mathbb{P} \big(\big\{ \omega\in\Omega: \textstyle \int_0^t \mathds 1_{ \{X_u \in A\}} \ \rmd u = 0  \big\}\big) = 1 \quad \text{for all }t >0\,,
  \end{equation}
  i.e.~$A$ is visited only on a measure zero set with full probability.
  Consider a measurable function $g: [0,1] \to [0,1]$ such that $g = 0$ on $[0,1] \setminus A$. Then
  \begin{equation*}
  \int_0^t g(X_u) \,\rmd W_u = 0 \quad \text{almost surely for all }t > 0\,.
  \end{equation*}
\end{lemma}
\begin{proof}
The statement follows directly from It\^{o}'s isometry
\begin{equation*}
\mathbb{E} \left[ \left( \int_0^t g(X_u) \rmd W_u \right)^2  \right] = \mathbb{E} \left[  \int_0^t g(X_u)^2 \rmd u \right] = \mathbb{E} \left[  \int_0^t \Big(g(X_u)^2 \mathds 1_{\{ X_u \in A \}} +  g(X_u)^2 \mathds 1_{ \{ X_u \in [0,1] \setminus A \}} \Big) \, \rmd u \right] = 0\,,
\end{equation*}
where the last equality follows immediately from (\ref{hittingtime}) and $g=0$ on $[0,1]\setminus A$. We conclude
\begin{equation*}
  \left(\int_0^t g(X_u) \,\rmd W_u \right)^2 =0  \quad  \text{almost surely}
\end{equation*}
due to nonnegativity, and the claim follows.
\end{proof}

\begin{proposition} \label{Nullset}
  Let $f'$ denote the weak derivative of $f$ as given by (\ref{fdef}). Then the choice of representative of $f'$ by determining $f'(\frac{1}{2})$ and $f'(0)$ does not affect the solution to the variational equation (\ref{varEqu}).
\end{proposition}
\begin{proof}
First, we show that
\begin{equation*}
 \mathbb{P} \big(\big\{ \omega\in\Omega: \textstyle \int_0^t \mathds 1_{ \{\vartheta_u = 1/2 \}} \ \rmd u = 0  \big\}\big) = 1 \quad \text{for all }t >0
\end{equation*}
by assuming the contrary to obtain a contradiction. As $\vartheta$ is a continuously differentiable process, this implies that $\vartheta_u = \frac{1}{2}$ for $u \in [t^*, t^* + \epsilon]$ for some $t^* \in (0,t)$ and $\epsilon >0$ with positive probability. This leads to $y(u) = -\frac{1}{b} \mod 1$ for $u \in (t^*, t^* + \epsilon)$ with positive probability. However, this implies that the continuous process $y_u$ for $u \in (t^*, t^* + \epsilon)$ given by
\begin{equation*}
  \rmd y = - \alpha y \,\rmd u + \sigma \,\rmd W_u
\end{equation*}
is constant with positive probability. This contradicts its definition as an Ornstein--Uhlenbeck process. The same reasoning obviously holds for $\theta = 0$.

Let $f_1' = f_2' = f'$ on $(0,1) \setminus \{\frac{1}{2}\}$ and assign arbitrary values at $\frac{1}{2}$ and $0$. Define
\begin{align*}
\rmd v &= \begin{pmatrix}
- \alpha & 0 \\ b & 0
\end{pmatrix} v \,\rmd t + \begin{pmatrix}
0 & \sigma f_1^{'}(\vartheta) \\ 0 & 0
\end{pmatrix} v \circ \rmd W_t^1\,,\\
d w &= \begin{pmatrix}
- \alpha & 0 \\ b & 0
\end{pmatrix} w \,\rmd t + \begin{pmatrix}
0 & \sigma f_2^{'}(\vartheta) \\ 0 & 0
\end{pmatrix} w \circ \rmd W_t^1\,.
\end{align*}
We apply Lemma~\ref{zerointegral} by choosing $X_u = \vartheta_u$ and $g(\vartheta_u) = f_1'(\vartheta_u)- f_2'(\vartheta_u) $ to conclude that
\begin{equation*}
\int_0^t  f_1'(\vartheta_u)\, \rmd W_u  = \int_0^t  f_2'(\vartheta_u) \,\rmd W_u \quad  \text{almost surely}.
\end{equation*}
As we do not have an It\^{o}--Stratonovich correction in this case, we can infer that $v_t = w_t$ almost surely for all $t>0$.
\end{proof}
We view $f'$ in the weak sense, disregarding the points $\frac{1}{2}$ and $0$, and we define $f'(\vartheta) = \sign(\frac{1}{2}- \vartheta)$, where
\begin{equation*}
\sign (x) = \begin{cases} 1 \ &\text{if} \ x \geq 0\,, \\ -1 \ &\text{if} \ x < 0\,.
\end{cases}
\end{equation*}
By Proposition~\ref{Nullset},  $D \varphi (t, \omega, x)$  does not depend on the choice of $f'(\frac{1}{2})$, so the variational equation (\ref{varEqu}) becomes
\begin{equation} \label{specvarEqu}
\rmd v = \begin{pmatrix}
- \alpha & 0 \\ b & 0
\end{pmatrix} v \,\rmd t + \begin{pmatrix}
0 & \sigma \ \sign(\frac{1}{2} - \vartheta_t) \\ 0 & 0
\end{pmatrix} v \circ \rmd W_t^1\,.
\end{equation}
We can now derive the following formula for the first Lyapunov exponent under assumption \eqref{sumcondition}, including $m=1$ with $f$ given by $\eqref{fdef}$:
\begin{proposition} \label{Novelty}
The top Lyapunov exponent of system (\ref{model1}) with $ \sum_{i=1}^m f_i'(\vartheta)^2 =1$
is given by
\begin{equation} \label{SpecLyap}
\lambda_1 = \int_0^{\pi} q(\phi) p(\phi)\,\rmd \phi\,,
\end{equation}
where $q(\phi):=  - \alpha \cos^2 \phi + b \cos \phi \sin \phi + \frac{1}{2}\sigma^2 (1 - 2 \cos^2 \phi)\sin^2 \phi$, and $p(\phi)$ is the solution of the stationary Fokker--Planck equation $\mathcal{L}^* p = 0$. $\mathcal{L}^*$ is the formal $L^2$-adjoint of the generator $\mathcal{L}$, which is given by
\begin{equation}  \label{generator}
\mathcal{L}g(\phi)  = \left( d(\phi) + \frac{1}{2} \tilde c(\phi) \tilde c '(\phi) \right) g'(\phi) + \frac{1}{2} \tilde c^2(\phi) g''(\phi)\,,
\end{equation}
where $d=d(\phi)$ is defined as in \eqref{cd}, and $\tilde c(\phi):=\sigma \,\sin^2 \phi$. Hence, $\lambda_1$ is identical to the top Lyapunov exponent of the linear system
\begin{equation} \label{specvarEqusim}
\rmd v = \begin{pmatrix}
- \alpha & 0 \\ b & 0
\end{pmatrix} v \,\rmd t + \begin{pmatrix}
0 & \sigma \\ 0 & 0
\end{pmatrix} v \circ \rmd W_t^1\,.
\end{equation}
\end{proposition}
\begin{proof}
Consider the SDE for the process $\phi(t)$ in It\^o form
\begin{equation*}
\rmd \phi = r(\phi)\rmd t + \sum_{i=1}^m c_i(\phi, \vartheta) \rmd W_t^i\,,
\end{equation*}
where
\begin{equation*}
r(\phi) = d(\phi) + \frac{1}{2} \sum_{i=1}^m  c_i(\phi, \vartheta) c_i'(\phi, \vartheta)= d(\phi) + \frac{1}{2}\sum_{i=1}^m  f_i'(\vartheta)^2 c_i(\phi) c_i'(\phi)  =d(\phi) + \frac{1}{2} \tilde c(\phi) \tilde c'(\phi)\,.
\end{equation*}
Furthermore, recall that by assumption~\eqref{sumcondition}
$$ \sum_{i=1}^m c_i^2(\phi, \vartheta) = \tilde c^2(\phi) \,.$$
As the coefficients of the SDE are smooth in $\phi$, we consider the kinetic equation for the probability density function of the process $\phi(t)$ (cf.~\cite{s73})
\begin{equation*}
\frac{\partial p(\psi,t)}{\partial t} = \sum_{n=1}^{\infty} \frac{(-1)^n}{n !} \frac{\partial^n}{\partial \psi^n} [a_n(\psi,t)p(\psi,t)]\,,
\end{equation*}
where
\begin{equation*}
a_n(\psi,t) = \lim_{\Delta t \to 0} \frac{1}{\Delta t} \mathbb{E}\left[ (\phi(t + \Delta t) - \phi(t))^n \vert \phi(t) = \psi  \right] \fa n\in\N\,.
\end{equation*}
Pick $\Delta t$ small, denote $ \Delta W_t^i = W^i(t + \Delta) - W^i(t)$ and recall that $\mathbb{E}[\Delta W_t^i] =0$ and $\mathbb{E}[(\Delta W_t^i)^2] =\Delta t$ for all $i=1, \dots, m$. Observe that
\begin{equation*}
 \phi(t + \Delta t) - \phi(t) =  r(\phi(t)) \Delta t + \sum_{i=1}^m c_i(\phi(t), \vartheta(t)) \Delta W_t^i + o(\Delta t)\,,
\end{equation*}
and
\begin{align*}
\left(  \phi(t + \Delta t) - \phi(t) \right)^2 &= r^2(\phi(t)) (\Delta t)^2 + 2  \sum_{i=1}^m r(\phi(t))  c_i(\phi(t), \vartheta(t)) \Delta W_t^i \Delta t \\
&+ \sum_{i=1}^m c_i^2(\phi(t), \vartheta(t)) (\Delta W_t^i)^2 + \sum_{i,j=1, i\neq j}^m c_i(\phi(t), \vartheta(t)) c_j(\phi(t), \vartheta(t)) (\Delta W_t^j)(\Delta W_t^i) + o(\Delta t)\,.
\end{align*}
Since the $\Delta W_t^i$ and $\Delta W_t^j$ are independent from each other for $i \neq j$ and are all independent from $\phi(t)$ and $\vartheta(t)$, we obtain that
\begin{displaymath}
a_1(\psi,t) = r(\psi) \quad\text{and}\quad a_2(\psi,t) = \tilde c^2(\psi)\,.
\end{displaymath}
We can see immediately from above that $a_n(\psi,t) = 0$ for $n \geq 3$. This proves \eqref{generator}, and \eqref{SpecLyap} follows from \eqref{GeneralLyap}. It follows from the calculations that formula~\eqref{SpecLyap} also gives the top Lyapunov exponent for system~\eqref{specvarEqusim}.
\end{proof}

The following statement is now a direct corollary of \cite[Theorem 3]{IL2001}.
\begin{theorem}\label{theol1l2}
Consider the stochastic differential equation \eqref{model1}, where the function $f$ is of the form \eqref{fdef}. Then the two Lyapunov exponents are given by
\begin{align}
\lambda_1(\alpha,b, \sigma) &= - \frac{\alpha}{2} + \frac{ \left|b \sigma\right|}{2} \int_{0}^{\infty}  v \ m_{\sigma, b, \alpha}(v)  \,\rmd v\,, \label{topLyap}\\
\lambda_2(\alpha,b, \sigma) &= - \frac{\alpha}{2} - \frac{\left|b \sigma\right|}{2} \int_{0}^{\infty}  v \ m_{\sigma, b, \alpha}(v)  \,\rmd v\,. \label{secLyap}
\end{align}
where
\begin{equation}\label{msba}
m_{\sigma, b, \alpha}(v)= \frac{ \frac{1}{\sqrt{v}} \exp \left( - \frac{\left|b \sigma\right|}{6} v^3 + \frac{\alpha^2}{2\left|b \sigma\right|} v \right)  } {\int_{0}^{\infty} \frac{1}{\sqrt{u}}\exp \left( - \frac{\left|b \sigma\right|}{6} u^3 + \frac{\alpha^2}{2\left|b \sigma\right|} u \right)  \rmd u}.
\end{equation}
\end{theorem}
\begin{proof}
Replacing $v = \begin{pmatrix}
v_1 \\ v_2
\end{pmatrix}$
by
$\hat v = \begin{pmatrix}
v_2 \\ \frac{v_1}{\sigma}
\end{pmatrix}$ leaves the Lyapunov exponents invariant and transforms \eqref{specvarEqusim} into the equation
 \begin{equation} \label{specvarnew}
\rmd v = \begin{pmatrix}
0& \sigma b \\ 0 & - \alpha
\end{pmatrix} v \,\rmd t + \begin{pmatrix}
0 & 0 \\ 1 & 0
\end{pmatrix} v \circ \rmd W_t^1\,.
\end{equation}
The matrices in this equation satisfy the assumptions of \cite[Theorem 3]{IL2001} which gives formulas~\eqref{topLyap} and~\eqref{secLyap}.
\end{proof}

\section{Bifurcation from negative to positive top Lyapunov exponent} \label{bifurresult}
We now use Theorem~\ref{theol1l2} to prove Theorem~\ref{maintheorem}, which asserts that there is a bifurcation from negative to positive Lyapunov exponent for the stochastic differential equation \eqref{model1}.

\begin{theorem-non1}
Consider the SDE \eqref{model1} with $f_i$, $i=1,\dots,m$, satisfying condition~\eqref{sumcondition}. Then there is $c_0 \approx 0.2823$ such that for all $\alpha > 0$ and $b \neq 0$ the number  $ \sigma_{0}(\alpha, b) = \frac{\alpha^{3/2}}{c_0^{1/2} \left|b\right|} > 0$ is the unique value of $\sigma$ where the top Lyapunov exponent $\lambda_1(\alpha, b,\sigma)$ of (\ref{model1}) changes its sign. In more detail, we have
\begin{equation*}
\lambda_1(\alpha, b,\sigma) \begin{cases}
& < 0 \quad \text{if} \ 0 < \sigma < \sigma_0(\alpha, b)\,, \\
& = 0 \quad \text{if} \ \sigma = \sigma_0(\alpha, b)\,, \\
& > 0 \quad \text{if} \ \sigma > \sigma_0(\alpha, b)\,.
\end{cases}
\end{equation*}
\end{theorem-non1}

\begin{proof}
We fix $\alpha >0$ and $b\not= 0$.
Introducing the change of variables $v= \frac{\alpha}{\left|b \sigma\right|} u$ in~\eqref{msba}, we obtain
\begin{equation}\label{lambda1mainthm}
\lambda_1(\alpha, b, \sigma) = \frac{\alpha}{2} \left( \int_{0}^{\infty}  u \ \tilde{m}_{\sigma,b, \alpha}(u)  \,\rmd u -1 \right),
\end{equation}
where
\begin{equation*}
\tilde{m}_{\sigma,b , \alpha}(u)= \frac{\frac{1}{\sqrt{u}} \exp \left( -  \frac{\alpha^3}{\sigma^2 b^2} \left[ \frac{1}{6} u^3 - \frac{1}{2} u \right] \right)}{\int_0^{\infty} \frac{1}{\sqrt{w}} \exp \left( -  \frac{\alpha^3}{\sigma^2 b^2} \left[ \frac{1}{6} w^3 - \frac{1}{2} w \right] \right)\,\rmd w}\,.
\end{equation*}
Defining $c := \frac{\alpha^3}{\sigma^2 b^2}$, we observe that $\lambda_1(\alpha, b, \sigma)$ has the same sign as the function $G:(0, \infty) \to \mathbb{R}$ given by
\begin{equation}
G(c) := \int_{0}^{\infty} \left( \sqrt{u} - \frac{1}{\sqrt{u}} \right) \exp \left( -  c \left[ \frac{1}{6} u^3 - \frac{1}{2} u \right] \right) \rmd u\,.
\end{equation}
Using dominated convergence, we may interchange the order of differentiation and integration and consider
$$ G'(c) = \int_{0}^{\infty} h_1(u) h_2(u) \exp \left( c \, h_2(u) \right) \rmd u\,, \ h_1(u) = \left( \sqrt{u} - \frac{1}{\sqrt{u}} \right)\,, \ h_2(u) = -\frac{1}{6} u^3 + \frac{1}{2} u\,.$$
Note that $h_1 h_2$, and thereby the integrand, has positive sign on the interval $(1,\sqrt{3})$ and negative sign on $(0,1)$ and $(\sqrt{3}, \infty)$. Basic claculations show that we have
$ \left| h_1(1-\delta) \right| > h_1(1+\delta) \text{ and } \ h_2(1-\delta) > h_2(1+\delta)$
for all  $\delta \in (0,\sqrt{3}-1)$. It follows that
$$ G'(c) < \int_{2-\sqrt{3}}^{\sqrt{3}} h_1(u) h_2(u) \exp \left( c \, h_2(u) \right) \rmd u < 0 \fa c \in (0, \infty)\,.$$
Hence, $G$ is strictly decreasing. Furthermore, we observe that $G(c) \to \infty$ as $c \searrow 0$ (using monotone convergence on $[\sqrt{3}, \infty)$) and that $G(c) \to -\infty$ as $c \to \infty$ (using similar arguments as for $G'$ and monotone convergence on $(0, 2-\sqrt{3})$).

Combining these observations, we may conclude that there is a unique $c_0$ such that $G(c_0) = 0$, $G(c) > 0$ for all $c \in (0, c_0)$ and $G(c) < 0$ for all $c \in (c_0, \infty)$. This proves the claim with $\sigma_0(\alpha,b) = \frac{\alpha^{3/2}}{c_0^{1/2} \left|b\right|}$. Numerical integration gives $c_0 \approx 0.2823$.
\end{proof}

\begin{remark}
  As explained in the Introduction, the same result holds if we interchange the roles of $\sigma$ and $b$. This can be seen also directly from the proof above.
\end{remark}

\begin{remark}
The random dynamical system induced by \eqref{model1} has a random set attractor $\{\tilde A(\omega)\}_{\omega \in \Omega}$
(see \cite[Definition 14.3]{rk04} for a formal definition) for all parameter values,
as can be seen similarly to \cite{ks98}.
The disintegrations $\mu_{\omega}$ of the ergodic invariant measure $\mu$ are supported on the fibers $\tilde A(\omega)$. In fact, the measurable random compact set $\{A(\omega)\}_{\omega \in \Omega}$ with fibers $A(\omega) = \supp(\mu_{\omega}) \subset \tilde A(\omega)$ is a minimal (weak) random point attractor of \eqref{model1} by \cite{fgs16}[Proposition 2.20 (1)].

The fact that $\{A(\omega)\}_{\omega \in \Omega}$ is a singleton almost surely if $\lambda_1 <0$, follows from a slightly modified reasoning alongside \cite[Theorem 2.23]{fgs16} and its proof. In the case of $\lambda_1 >0$, we deduce that $\mu_{\omega}$ is atomless almost surely as in the proof of \cite[Remark 4.12]{b91}. Hence, Theorem~\ref{maintheorem} implies the bifurcation from an attractive random equilibrium to an atomless random point attractor (also called random strange attractor).
\end{remark}

The positive top Lyapunov exponent is the only characterization of chaos we can give in this case as an analysis in the sense of \cite{wy03} seems not feasible for white noise. However, the geometric mechanism of shear-induced chaos can still be understood along the same lines: the white noise drives some points on the limit cycle up and some down. Due to the phase amplitude coupling $b$, the points with larger $y$-coordinates move faster in the $\vartheta$-direction. At the same time, the dissipation force with strength $\alpha$ attracts the curve back to the limit cycles. This provides a mechanism for stretching and folding characteristic of chaos. The transition to chaos in the continuous time stochastic forcing is much faster than in the case of periodic kicks due to the effect of large deviations \cite{ly08}. This is due to the fact that points end up in areas with arbitrarily large values of $y$ with positive probability. Hence, not so much shear is needed to generate the described stretching and folding due to phase amplitude coupling. However, for very small shear and noise, the dissipation leads to sinks being formed between these large deviation events, and the attractor ends up to be a singleton.

In Figure~\ref{fig1234}, we show the top Lyapunov exponent as a function of $\sigma$ for fixed $b$ and $\alpha$ according to formula \eqref{topLyap}. We have used numerical integration up to machine precision to calculate $\lambda_1$. The bifurcations of the sign of $\lambda_1$ at $\sigma_0(\alpha,b)$ is clearly seen in Figures~\ref{fig_1}-\ref{fig_3}. Furthermore, note that $\lambda_1 \to 0$ from below for $\sigma \to 0$. The figures illustrate that $\sigma_{0}(\alpha, b)$ is an increasing function of $\alpha$ and a decreasing function of $b$, or differently phrased: the larger the proportion of shear to dissipation, $b/\alpha$,  the smaller the bifurcation point $\sigma_0(\alpha,b)$.  In Figure~\ref{fig_4}, we choose small values of $b$ and $\alpha$, but $b/\alpha$ large. We see no negative values of $\lambda_1$ as we would have to take values of $\sigma$ too small for the numerical integration.

We have chosen the same parameter regimes as in \cite{ly08}, where Lin and Young investigate numerically the Lyapunov exponents of the system given by
\begin{align} \label{youngmodel}
\begin{array}{r@{\;\,=\;\,}l}
\rmd y & - \alpha y \rmd t + \sigma \sin(2 \pi \vartheta) \circ \rmd W_t\,, \\
\rmd \vartheta & (1 + b y) \,\rmd t\,,
\end{array}
\end{align}
taking $\vartheta \in [0, 2 \pi]$. Their numerical results show exactly the same qualitative behaviour apart from a slightly different scaling due to the factor $2 \pi$. Note that our setting contains two approximations of model~\eqref{youngmodel}. The first option is to take also one Brownian motion, i.e. $m=1$ in~\eqref{model1}, and $f_1$ continuous and piecewise linear with constant absolute value of the derivative almost everywhere. The accordance of the numerics for~\eqref{youngmodel} and our results show that the simpler choice of the diffusion coefficient in our case does not change the qualitative behaviour. This is not a surprise, since we can derive formula (\ref{topLyap}) for $\lambda_1$ if we choose $f$ to be piecewise linear on the intervals $[\frac{i}{4}, \frac{i+1}{4}]$ for $i=0,1,2,3$ with $\left| f' \right|$ constant such that it represents a linear approximation of the sine function.
\begin{figure}[H]
\centering
\begin{subfigure}{.45\textwidth}
  \centering
  \includegraphics[width=1\linewidth]{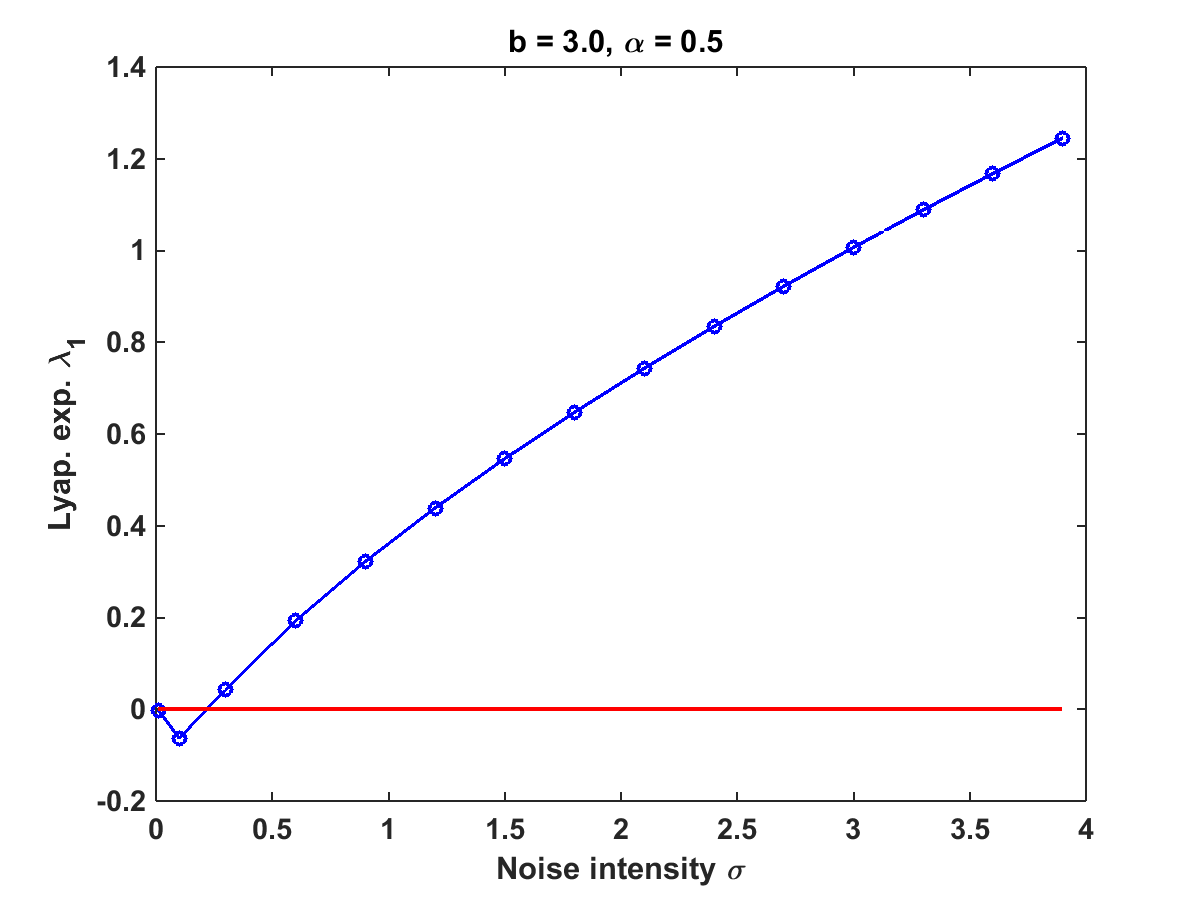}
   \caption{}
  \label{fig_1}
\end{subfigure}%
\begin{subfigure}{.45\textwidth}
  \centering
  \includegraphics[width=1\linewidth]{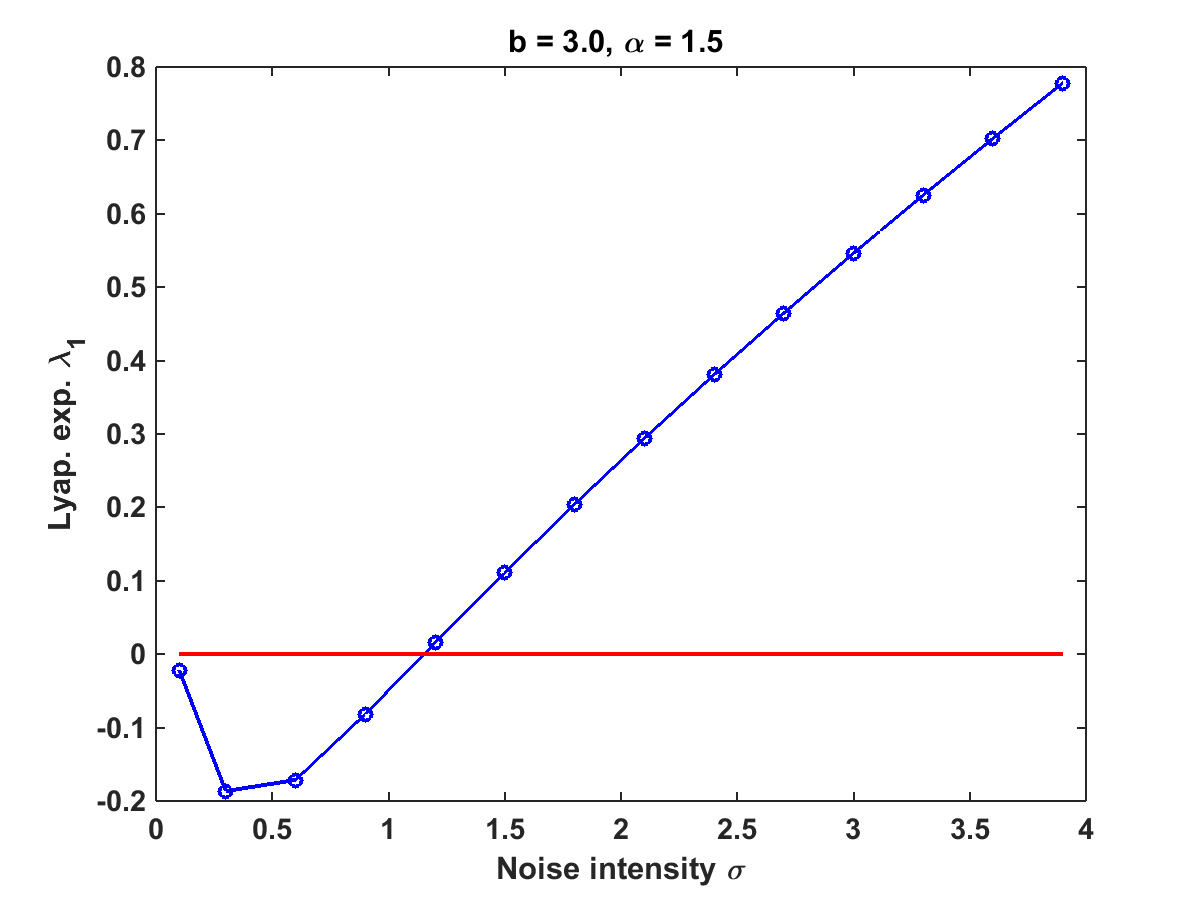}
   \caption{}
  \label{fig_2}
\end{subfigure}
\newline
\begin{subfigure}{.45\textwidth}
  \centering
  \includegraphics[width=1\linewidth]{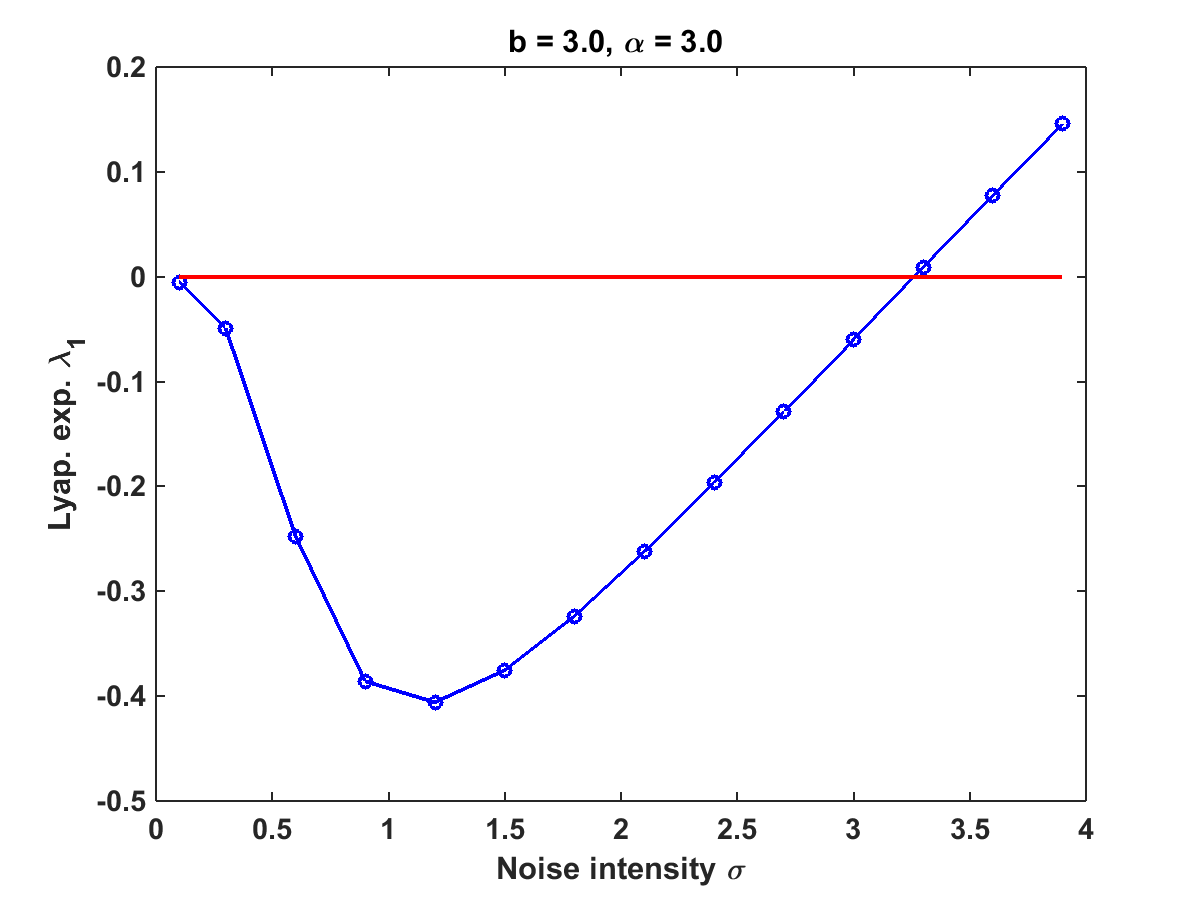}
   \caption{}
  \label{fig_3}
\end{subfigure}%
\begin{subfigure}{.45\textwidth}
  \centering
  \includegraphics[width=1\linewidth]{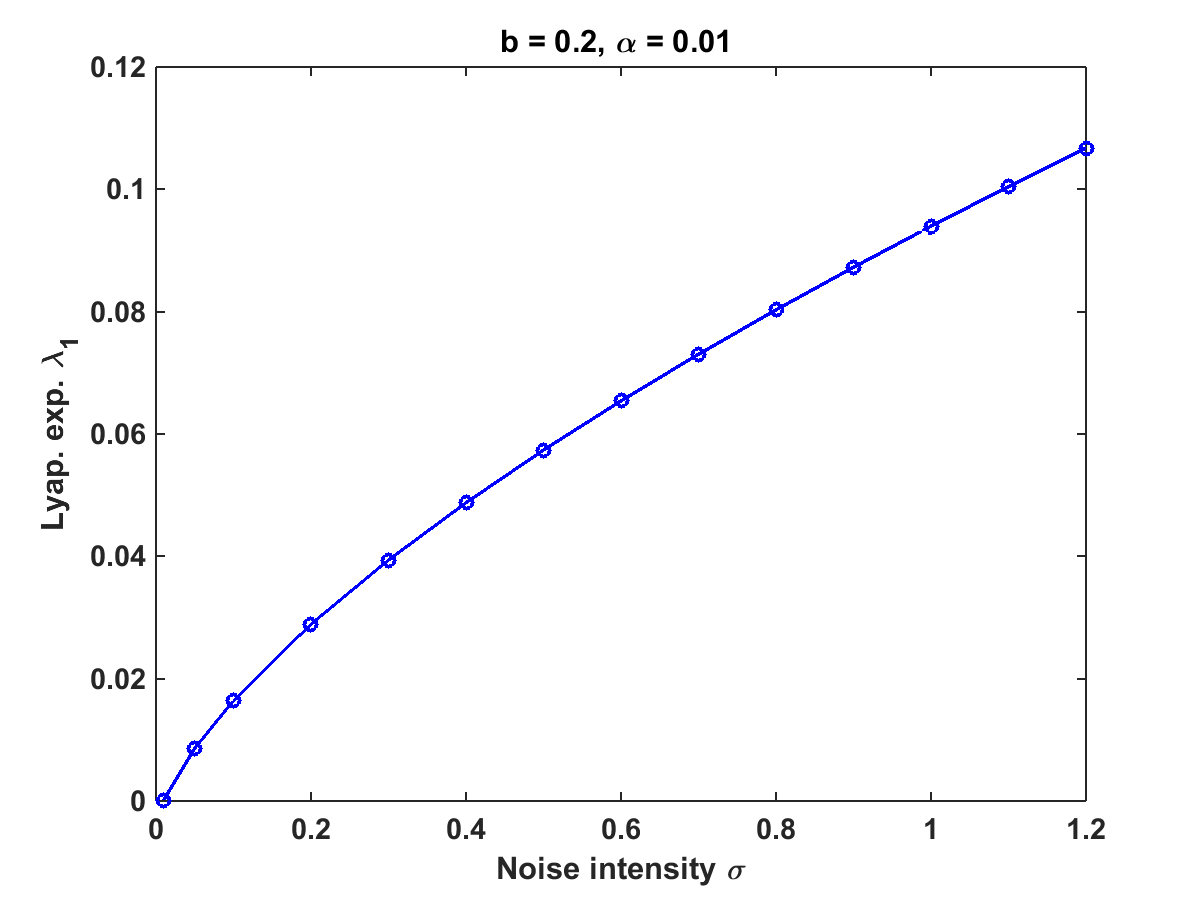}
  \caption{}
  \label{fig_4}
\end{subfigure}
\caption{The top Lyapunov exponent $\lambda_1$ as a function of $\sigma$ for fixed $b$ and $\alpha$. The dots indicate the values of $\lambda_1$ that were calculated according to \eqref{topLyap} using numerical integration. Figures~\ref{fig_1}-\ref{fig_3} illustrate that $\sigma_{0}(\alpha, b)$ increases monotonously in $\alpha$. In Figure~\ref{fig_4}, $b$ and $\alpha$ are small, but $b/\alpha$ is large. We don't see the transition to $\lambda_1 <0$ since we would have to take values of $\sigma$ too small for the numerical integration.}
\label{fig1234}
\end{figure}
The second option is given by~\eqref{fdef2} which adds an additional Brownian motion with the cosine function as diffusion coefficient. As we have seen, this slightly extended model does not change the qualitative behaviour and validates analytically Lin and Young's numerical investigations for model~\eqref{youngmodel}.

Furthermore, we confirm Lin and Young's conjecture concerning a scaling property of $\lambda_1$ with respect to the parameters. For model~\eqref{youngmodel} they observed numerically that, under the transformations $\alpha \mapsto k \alpha$, $b \mapsto k b$ and $\sigma \mapsto \sqrt{k} \sigma$, $\lambda_1$ transforms approximately as $\lambda_1 \mapsto k \lambda_1$. This scaling property holds exactly for our model~\eqref{model1}.
\begin{proposition}
Consider the stochastic differential equation \eqref{model1}, where the function $f$ is of the form \eqref{fdef}. Then the top Lyapunov exponent $\lambda_1$ as given by \eqref{topLyap} satisfies
\begin{equation} \label{scaling}
\lambda_1(k \alpha, k b, \sqrt{k} \sigma) = k \lambda_1(\alpha, b, \sigma) \quad \fa k \in \mathbb{R}^+ \setminus \{0\} \,.
\end{equation}
\end{proposition}
\begin{proof}
Recall from Proposition~\ref{Novelty} that $\lambda_1$ can be calculated as the top Lyapunov exponent of the linear system~\eqref{specvarEqusim}. Then the claim follows immediately if we conduct the time change $ t \mapsto kt$ in this equation.
\end{proof}
\section{Summary and outlook} \label{summary}
We have investigated systems with limit cycles on a cylinder perturbed by white noise. We were able to show a transition from negative to positive top Lyapunov exponents for fixed dissipation parameter $\alpha$ and big enough noise $\sigma$ and/or shear $b$. This implies a bifurcation of the random attractor from a random equilibrium to random strange attractor.

In the case of positive Lyapunov exponents, it remains an open problem to describe the attractor using concepts from ergodic theory, as entropy and SRB measures \cite{ly88, wy03}, in order to have a more rigorous notion of chaos.

The results of this paper may well be relevant to shed more light on the problem of stochastic Hopf bifurcation, where numerical studies indicate a transition from negative to positive Lyapunov exponent as explained in the Introduction.

\section*{Acknowledgments} The authors would like to thank the referee for carefully reading our manuscript and for giving such constructive comments which substantially helped improving the quality of the paper, in particular strengthening the main result. We also express exceptional gratitude to Martin Hairer and Nils Berglund for suggesting a very valuable generalization of the original model. In addition, the authors thank Alexis Arnaudon, Darryl Holm, Aleksandar Mijatovic, Nikolas N{\"u}sken, Grigorios Pavliotis and Sebastian Wieczorek for useful discussions. Maximilian Engel was supported by a Roth Scholarship from the Department of Mathematics at Imperial College London and the SFB Transregio 109 "Discretization in Geometry and Dynamcis" sponsored by the German Research Foundation (DFG). Jeroen S.W.~Lamb acknowledges the support by Nizhny Novgorod University through the grant RNF 14-41-00044, and Martin Rasmussen was supported by an EPSRC Career Acceleration Fellowship EP/I004165/1. This research has also been supported by EU Marie-Curie IRSES Brazilian-European Partnership in Dynamical Systems (FP7-PEOPLE-2012-IRSES 318999 BREUDS) and EU Marie-Sk\l odowska-Curie ITN Critical Transitions in Complex Systems (H2020-MSCA-2014-ITN 643073 CRITICS).


\bibliographystyle{plain}

\bibliography{mybibfile}

\begin{thebibliography}{10}

\bibitem{acd16}
A.~Arnaudon, A.L. De~Castro, and D.D. Holm.
\newblock Noise and dissipation on coadjoint orbits.
\newblock {\em arXiv1601.02249[math.PR]}, 2016.

\bibitem{a98}
L.~Arnold.
\newblock {\em Random Dynamical Systems}.
\newblock Springer, Berlin, 1998.

\bibitem{ass96}
L.~Arnold, N.~Sri~Namachchivaya, and K.~R. Schenk-Hopp\'{e}.
\newblock Toward an understanding of stochastic {H}opf bifurcation: A case
  study.
\newblock {\em International Journal of Bifurcation and Chaos},
  6(11):1947--1975, 1996.

\bibitem{b91}
P.H. Baxendale.
\newblock Statistical equilibrium and two-point motion for a stochastic flow of
  diffeomorphisms.
\newblock In {\em Spatial stochastic processes}, volume~19 of {\em Progress in
  Probability}, pages 189--218. Birkh\"auser Boston, Boston, MA, 1991.

\bibitem{b94}
P.H. Baxendale.
\newblock A stochastic {H}opf bifurcation.
\newblock {\em Probability Theory and Related Fields}, 99:581--616, 1994.

\bibitem{b04}
P.H. Baxendale.
\newblock Stochastic averaging and asymptotic behavior of the stochastic
  {D}uffing-van der {P}ol equation.
\newblock {\em Stochastic Processes and their Applications}, 113:235--272,
  2004.

\bibitem{bg02}
P.H. Baxendale and L.~Goukasian.
\newblock Lyapunov exponents for small perturbations of {H}amiltonian systems.
\newblock {\em Annals of Probability}, 30:101--134, 2002.

\bibitem{cdlr16}
M.~Callaway, T.S. Doan, J.S.W. Lamb, and M.~Rasmussen.
\newblock The dichotomy spectrum for random dynamical systems and pitchfork
  bifurcations with additive noise.
\newblock {\em Annales de l'Institut Henri Poincar\'e, Probabilit\'es et
  Statistiques}, 53(4):1548--1574, 2017.

\bibitem{cf98}
H.~Crauel and F.~Flandoli.
\newblock Additive noise destroys a pitchfork bifurcation.
\newblock {\em Journal of Dynamics and Differential Equations}, 10(2):259--274,
  1996.

\bibitem{dsr11}
L.~Deville, N.~Sri~Namachchivaya, and Z.~Rapti.
\newblock Stability of a stochastic two-dimensional non-{H}amiltonian system.
\newblock {\em Journal of Applied Mathematics}, 71(4):1458--1475, 2011.

\bibitem{delr17}
T.S. Doan, M.~Engel, J.S.W. Lamb, and M.~Rasmussen.
\newblock Hopf bifurcation with additive noise.
\newblock {\em arXiv:1710.09649v1[math.DS]}, 2007.

\bibitem{fgs16}
F.~Flandoli, B.~Gess, and M.~Scheutzow.
\newblock Synchronization by noise.
\newblock {\em Probability Theory and Related Fields}, 168(3--4):511--556,
  2017.

\bibitem{IL99}
P.~Imkeller and C.~Lederer.
\newblock An explicit description of the {L}yapunov exponents of the noisy
  damped harmonic oscillator.
\newblock {\em Dynamics and Stability of Systems}, 14(4):385--405, 1999.

\bibitem{IL2001}
P.~Imkeller and C.~Lederer.
\newblock Some formulas for {L}yapunov exponents and rotation numbers in two
  dimensions and the stability of the harmonic oscillator and the inverted
  pendulum.
\newblock {\em Dynamical Systems}, 16(1):29--61, 2001.

\bibitem{ks98}
H.~Keller and B~Schmalfuss.
\newblock Attractors for stochastic differential equations with nontrivial
  noise.
\newblock {\em Buletinul A.S. a R.M. Matematica}, 26(1):43--54, 1998.

\bibitem{rk04}
P.E. Kloeden and M.~Rasmussen.
\newblock {\em Nonautonomous dynamical systems}, volume 176 of {\em
  Mathematical Surveys and Monographs}.
\newblock American Mathematical Society, Providence, RI, 2011.

\bibitem{Lamb_15_1}
J.S.W. Lamb, M.~Rasmussen, and C.S. Rodrigues.
\newblock Topological bifurcations of minimal invariant sets for set-valued
  dynamical systems.
\newblock {\em Proceedings of the American Mathematical Society}, 143(9), 2015.

\bibitem{lj87}
Y.~Le~Jan.
\newblock Equilibre statistique pour les produits de diffeomorphismes
  aleatoires independants.
\newblock {\em Annales de l'Institut Henri Poincar\'e, Probabilit\'es et
  Statistiques}, 23(1):111--120, 1987.

\bibitem{ly88}
F.~Ledrappier and L.-S. Young.
\newblock Entropy formula for random transformations.
\newblock {\em Probability Theory and Related Fields}, 80:217--240, 1988.

\bibitem{ls12}
Z.~Lian and M.~Stenlund.
\newblock Positive {L}yapunov exponent by a random perturbation.
\newblock {\em Dynamical Systems}, 27(2):239--252, 2012.

\bibitem{ly08}
K.~K. Lin and L.-S. Young.
\newblock Shear-induced chaos.
\newblock {\em Nonlinearity}, 21:899--922, 2008.

\bibitem{nj15}
J.~Newman.
\newblock Necessary and sufficient conditions for stable synchronisation in
  random dyncamical systems.
\newblock {\em arXiv1408.5599v3[math.PR]}, 2015.

\bibitem{nj16}
J.~Newman.
\newblock Synchronisation of almost all trajectories of a random dynamical
  system.
\newblock {\em arXiv1511.08831v2[math.PR]}, 2015.

\bibitem{os10}
W.~Ott and M.~Stenlund.
\newblock From limit cycles to strange attractors.
\newblock {\em Communications in Mathematical Physics}, 296(1):215--249, 2010.

\bibitem{sh96}
K.~R. Schenk-Hopp\'{e}.
\newblock Bifurcation scenarios of the noisy {D}uffing-van der {P}ol
  oscillator.
\newblock {\em Nonlinear Dynamics}, 11:255--274, 1996.

\bibitem{s73}
T.T. Soong.
\newblock {\em Random differential equations in science and engineering}.
\newblock Mathematics in Science and Engineering 103. Academic Press, New York
  and London, 1973.

\bibitem{wy02}
Q.~Wang and L.-S. Young.
\newblock From invariant curves to strange attractors.
\newblock {\em Communications in Mathematical Physics}, 225(2):275--304, 2002.

\bibitem{wy03}
Q.~Wang and L.-S. Young.
\newblock Strange attractors in periodically-kicked limit cycles and {H}opf
  bifurcations.
\newblock {\em Communications in Mathematical Physics}, 240(3):509--529, 2003.

\bibitem{w09}
S.~Wieczorek.
\newblock Stochastic bifurcation in noise-driven lasers and {H}opf oscillators.
\newblock {\em Physical Review E}, 79:1--10, 2009.

\bibitem{y08}
L.-S. Young.
\newblock Chaotic phenomena in three settings: large, noisy and out of
  equilibrium.
\newblock {\em Nonlinearity}, 21:245--252, 2008.

\bibitem{Zmarrou_07_1}
H.~Zmarrou and A.J. Homburg.
\newblock Bifurcations of stationary measures of random diffeomorphisms.
\newblock {\em Ergodic Theory and Dynamical Systems}, 27(5):1651--1692, 2007.

\end{thebibliography}

\end{document}